\documentclass[11pt,a4paper,oneside]{amsart}
\usepackage{graphicx}
\usepackage{tikz}
\usetikzlibrary{patterns}
\usepackage[utf8]{inputenc}

\usepackage{amsmath}
\usepackage{amssymb}
\usepackage{amsthm}
\usepackage[square,numbers]{natbib}

\usepackage{xifthen}

\usepackage{aliascnt}
\usepackage[bookmarksopen,bookmarksdepth=2,backref=none,hidelinks]{hyperref}

\theoremstyle{definition}
\newtheorem{definition}{Definition}[section]
\newaliascnt{lemma}{definition}
\newaliascnt{theorem}{definition}
\newaliascnt{corollary}{definition}
\newaliascnt{proposition}{definition}
\newaliascnt{assumption}{definition}
\newaliascnt{remark}{definition}
\newaliascnt{example}{definition}
\newaliascnt{conjecture}{definition}
\usepackage{esint}

\newtheoremstyle{remark2}{}{}{}{}{\bfseries}{.}{ }{}
\theoremstyle{remark2}

\newtheorem{remarkx}[remark]{Remark}
\newenvironment{remark}
  {\pushQED{\qed}\remarkx}
  {\popQED\endremarkx}

\theoremstyle{plain}
\newtheorem{lemma}[lemma]{Lemma}
\newtheorem{theorem}[theorem]{Theorem}

\newtheorem{proposition}[proposition]{Proposition}

\newcommand{\nablasym}{{\varepsilon}}

\aliascntresetthe{lemma}
\aliascntresetthe{theorem}
\aliascntresetthe{corollary}
\aliascntresetthe{proposition}
\aliascntresetthe{remark}
\aliascntresetthe{conjecture}
\aliascntresetthe{assumption}

\newcommand{\R}{\ensuremath{\mathbb{R}}}

\newcommand{\N}{\ensuremath{\mathbb{N}}}

\newcommand{\abs}[1]{\left|{#1}\right|}
\newcommand{\norm}[2][]{\left\|#2\right\|_{{#1}}} 
\newcommand{\inner}[3][]{\left\langle #2,#3 \right \rangle_{#1}} 

\DeclareMathOperator{\diver}{div}
\DeclareMathOperator{\cof}{cof}
\DeclareMathOperator{\tr}{tr}

\DeclareMathOperator{\Lip}{Lip}
\newcommand{\bog}{\mathcal{B}}

\numberwithin{equation}{section}

\newcommand{\seb}[1]{{\color{blue}#1}}

\usepackage{todonotes}
\allowdisplaybreaks

\title[Variational methods for FSI]{Variational methods for fluid-structure interaction and porous media}
\author{B.~Bene\v{s}ov\'{a}, M.~Kampschulte, S.~Schwarzacher}
\address{Department of Mathematics and Physics, Charles University Prague,  \nolinkurl{benesova@karlin.mff.cuni.cz}, \nolinkurl{kampschulte@karlin.mff.cuni.cz}, \nolinkurl{schwarz@karlin.mff.cuni.cz} }
\date{\today}

\begin{document}
 \begin{abstract}
 In this work we consider a poroelastic flexible material that may deform largely which is situated in an incompressible fluid driven by the Navier-Stokes equations in two or three space dimensions. By a variational approach we show existence of weak solutions for a class of such coupled systems. We consider the unsteady case, this means that the PDE for the poroelastic solid involves the Frechet derivative of a non-convex functional as well as (second order in time) inertia terms. 
 \end{abstract}

 \maketitle

 \section{Introduction}

\emph{Poromechanics} or \emph{mechanics of porous media} has been a lively area of research in engineering and continuum mechanics (see e.g. the monographs \cite{de2006trends,coussy2004poromechanics}) as the possible applications of the proposed theories are numerous, ranging from the mechanics of tissue and biological materials in general (see e.g.\ \cite{ogden2006mechanics}) to the recently important breathing through masks. In many of these cases, the poroelastic matrix is not isolated but instead immersed into a fluid which may flow through the elastic structure. Naturally, then, the porous medium and the free fluid are interacting with each other. Indeed on the mutual interface they both exert pressure and stresses to each other and on the top of that the fluid may enter the pores of the structure and flow through it. Such a setting is referred to as \emph{fluid-pororelastic structure interaction (FPSI)} and will be the subject of the present work. The main result of this paper is the well-posedness for a class of FPSI  problems allowing large deformations and involving incompressible fluids driven by the Navier-Stokes equations in two or three space dimensions.

In order to predict the response of the FPSI-system, we need to devise a model for the free fluid, the poroelastic structure and prescribe suitable interface conditions. Each of the tasks can be approached on several scales, ranging from the microscopic to the macroscopic one. For the former the skeleton is modeled in detail and a (yet another) fluid-structure interaction problem between the skeleton and the circumflowing fluid is set up. More common in the engineering literature, however, is a more \emph{macroscopic approach} where the porous medium is assumed to be saturated by the fluid, i.e.\ they form a kind of ``mixture'' in each material point and an \emph{ad hoc} form of the stress tensor is prescribed for the porous medium \cite{chapelle2014general}. Such an approach covers the early and famous models for the porous medium due to Darcy \cite{darcy1856fontaines}, Brinkman \cite{brinkman1949calculation} or Biot \cite{biot1941general} and will be also the one we follow in this contribution. Let us also note that a justification of the proposed models by homogenization is highly  relevant indeed and some results concerning the aforementioned standard models for porous media have already been obtained, see e.g. \cite{All90II,Conca1,hornung1996homogenization}.

Following the macroscopic approach, we will assume that the porous medium is fully saturated by the fluid and can be completely described by the fluid velocity $v$ and the deformation of the solid $\eta$ (see Section \ref{sec:mod} for a more detailed modeling description). We will take a \emph{variational point of view} for modeling as well as for the analysis by defining the constitutive behavior of the mechanical system by prescribing its \emph{stored energy} as well as \emph{dissipation functional} (dissipation pseudo-potential). Such an approach has been advocated by many authors (e.g.\ \cite{halphenMateriauxStandardGeneralises1975} for solids or \cite{ziegler.h:some,ziegler.h.wehrli.c:derivation, ziegler.h.wehrli.c:on} for fluids) as it may simplify the modeling.  On the top of that, we argue in Section \ref{sec:2} that the variational approach is essential also from the point of view of \emph{mathematical analysis} if a \emph{large-strain, non-linear} model is to be used for the involved solid. Indeed, in such a case, the Helmholtz free energy cannot be chosen as a convex functional for physical restrictions and thus the resulting differential operator is not monotone.

In the case of FPSI, the variational approach may also facilitate the derivation of boundary conditions at the interface of the poroelastic medium itself and the fluid into which it is immersed, see e.g.\ \cite{dell2009boundary}. Prescribing the above mentioned energy and dissipation for the whole system then automatically yields a balance of stresses at this boundary and, additionally, imposes a ``weak continuity'' condition for the fluid velocity over the porous media interface, as we ask for its derivatives to be globally integrable. Nonetheless, generalizations are, of course, possible. For this see in particular the notes at the end of Section 2. 

While the engineering literature on models in poroelasticity and FPSI seems to be rich, the mathematical literature studying well-posedness of the FPSI is, to the authors' knowledge, rather scarce. Interaction of the linear or non-linear poroelastic medium with Stokes flow has been studied in \cite{showalter2005poroelastic,ambartsumyan2019nonlinear,bociu2020multilayered}. Moreover, the coupling of the Stokes and Darcy flow has been studied (see e.g. \cite{wilbrandt2019stokes} for an overview) but here, intrinsically, the region occupied by the porous medium is fixed a-priori. Similarly, the interaction of the Biot-poroelastic medium with a Navier-Stokes flow was studied in \cite{cesmelioglu2017analysis}, but again with the interface wall between the porous solid and fluid fixed. Thus, well-posedness of a FPSI problem featuring a fully non-linear poroelastic solid that may undergo large deformations coupled to a Navier-Stokes fluid has not been addressed in literature so far. In this contribution, we work exactly in this setting and we establish \emph{existence of weak solutions}. In order to do so, it is crucial to approach the modeling as well as the subsequent analysis in an \emph{energetic (or variational)} way to access methods from the calculus of variations. We provide an overview of the analytical details of the approach in Section \ref{sec:2} but would like to emphasize here that working the proposed framework not only allows to prove the result presented here but also paves the road for studying further FPSI models. Indeed, models including dependence of the stresses on e.g. the porosity, the porous pressure of the fluid or more general coupling conditions are accessible to the analysis, though a careful investigation will be needed.

This work is built up as follows: In Section \ref{sec:mod} we present the model considered in this work and formulate the main results. In Section \ref{sec:2} we overview the ingredients of the variational approach and in Section \ref{sec:4} we present the proof of the main result, Theorem \ref{thm:main}.

\section{Modeling and main results} 
\label{sec:mod}
  
To set up the porous media problem, we consider a regular enough\footnote{The regularity of the boundary of $\Omega$ as well as $Q$ only fully comes into play when discussing contact of the solid with $\partial \Omega$. This is discussed in a bit more detail in \cite{BenKamSch20}, and will be thoroughly discussed in a forthcoming paper involving the second author. For the purpose of the paper at hand, we ignore this aspect completely and just require enough regularity to meaningfully talk about boundary values.} spatial domain $\Omega \subset \R^n$ in which both the porous medium an the circumlying fluid are contained. 
	
 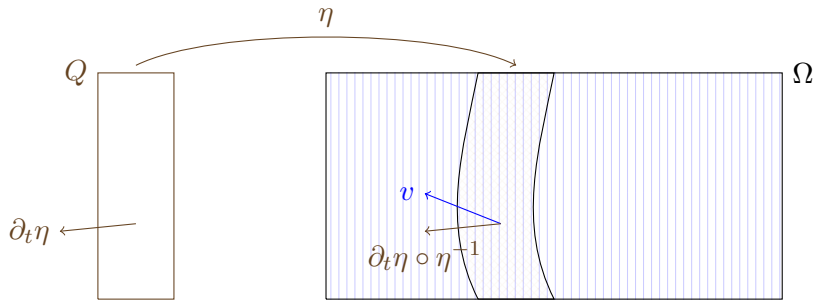
\begin{figure}[ht]
  
 \begin{center}
 \definecolor{darkbrown}{RGB}{101,67,33}
 \begin{tikzpicture}
  \draw[color=darkbrown] (0,0) -- (1,0) -- (1,3) -- (0,3) node[left] {$Q$} -- (0,0);

  \fill[pattern=vertical lines,pattern color=blue, opacity =.2] (3,0) -- (9,0) -- (9,3) node[right] {$\Omega$} -- (3,3) --(3,0);
  \draw (3,0) -- (9,0) -- (9,3) node[right] {$\Omega$} -- (3,3) --(3,0);
  \fill[pattern=crosshatch,pattern color=brown, opacity=.2] (5,0) .. controls +(-.5,1) and +(-.2,-1) .. (5,3) -- (6,3) .. controls +(-.2,-1) and +(-.5,1) .. (6,0)--(5,0);
  \draw (5,0) .. controls +(-.5,1) and +(-.2,-1) .. (5,3) -- (6,3) .. controls +(-.2,-1) and +(-.5,1) .. (6,0)--(5,0);
  
  \draw[->,color=darkbrown] (.5,3.1) .. controls +(1,.5) and +(-1,.5) .. node[above] {$\eta$} (5.5,3.1);
  
  \draw[->,color=blue] (5.3,1) -- +(-1,.4) node [left] {$v$};
  \draw[->,color=darkbrown] (5.3,1) -- +(-1,-.1) node[below] {$\partial_t \eta \circ \eta^{-1}$};
  \draw[->,color=darkbrown] (.5,1) -- + (-1,-.1) node[left] {$\partial_t \eta$};
 \end{tikzpicture}
 \end{center}
 \caption{A short illustration of the geometry involved in the model.}
 \end{figure}
 
As already announced, we assume that the poroelastic medium is a perfectly homogenized mixture of the solid and the fluid in each material point while outside the porous medium the fluid is found in its ``pure" state. Thus, the fluid is allowed to flow through the whole of $\Omega$, including the part already occupied by the solid. The state of the solid will be described by its deformation, a map $\eta:Q\to \Omega$; i.e.\ from some regular enough reference domain $Q \subset \R^n$ to the container $\Omega$. We choose to fix ``boundary values'' $\eta|_P = \gamma$ for some subset $P \subset \overline{Q}$, but technically this is optional. The deformation $\eta$ additionally fixes the geometry of the porous medium $\eta(Q)$.

The fluid, on the other hand, is described by its velocity $v: \Omega \to \R^n$ and pressure $p: \Omega \to \R^n$ with $\diver v = 0$. In order to simplify the discussion, we will assume no-slip conditions at the boundary of $\Omega$, i.e.\ $v|_{\partial \Omega} = 0$. 
As the fluid is assumed incompressible, the pressure will later ``disappear'' from the weak formulation thanks to the choice of appropriate test functions.

Having defined all the quantities involved, we can now turn back to modeling by prescribing the stored energy. We assume an additive decomposition of the energy for the fluid and the solid. However, as the fluid is assumed to be purely Newtonian and incompressible, no energy can be stored in the fluid. Thus, the stored energy consists just of the solid one denoted as $E(\eta)$. The overall energy is then given by the sum
\begin{align*}
\int_\Omega \frac{\rho_f}{2} \abs{v}^2 dy + \int_Q \frac{\rho_s}{2} \abs{\partial_t \eta}^2 dx + E(\eta). 
\end{align*}
Here $\rho_f$ and $\rho_s$ are given mass densities of the fluid and solid respectively.

To complete the modeling we need to prescribe a form of $E(\eta)$. For simplicity, we stick to the the Saint Venant-Kirchhoff energy enriched with a term introducing resistance of the solids to infinite compression and a higher-order gradient regularisation:
\begin{align}
E(\eta) &:= \begin{cases}
 \int_Q \frac{1}{8}|\nabla \eta^T \nabla \eta-I|^2_\mathcal{C} + \frac{1}{(\det \nabla \eta)^a} +\frac{1}{q} \abs{\nabla^2 \eta }^q dx & \text{if $\det \nabla \eta > 0$ a.e. in $Q$ } \\
+ \infty & \text{otherwise} 
\end{cases} \label{st-venant}
\end{align}
where we used the notation $|\nabla \eta^T \nabla \eta-I|^2_\mathcal{C}:= \big(\mathcal{C} (\nabla \eta^T \nabla \eta- I) \big) \cdot \big( \nabla \eta^T \nabla \eta- I \big)$, with $\mathcal{C}$ being a fourth order positive definite tensor of elastic constants, $q>n$ and $a>\frac{qn}{q-n}$. 

\begin{remark}[Non-simple materials]
\label{non-simple}
Let us stress that the proposed energy functional for the solid depends not only on the \emph{deformation gradient}, which is the standard setting, but it also includes a term depending on the \emph{second gradient of the deformation}. Thus, the proposed model belongs to the class of so-called \emph{non-simple} materials that have been proposed by Toupin \cite{toupin1962elastic} and later further developed in other works, e.g. \cite{podio2010hypertractions,vsilhavy1985phase,fried2006tractions}. 
In the mathematical literature concerned with the large-deformation regime this concept is often employed to gain additional, necessary regularity. Also the present work relies on it. Indeed, at the heart of our variational approach (see also Section \ref{sec:2}) is that we will be able to solve a time-discrete variant of the set of PDEs in \eqref{strong} by solving a minimization problem and then relying on the fact that this minimizer solves the associate Euler-Langrage equation. However, as our energy functional takes infinite values, the relation between the functional and the Euler-Lagrange equation is not straighforward. Indeed, the minimizer may lie almost at the boundary of the domain of $E(\eta)$ so that taking variations at the minimizer would not be possible. This is indeed a well-known problem in mathematical elasticity \cite{Bal02} that has not found a satisfactory solution to date. Let us just mention that even in the one-dimensional setting explicit examples have been found in \cite{ball1987one} that illustrate this exact phenomenon: a minimizer of the functional exists but it does not satisfy the Euler-Lagrange equation.

To the authors' knowledge, as of today, the only setting where it is known that this phenomenon does not occur is the one of \emph{non-simple materials}. Indeed, Healey and Kr\"{o}mer have found in \cite{healeyInjectiveWeakSolutions2009} that if the indices $q$ and $a$ in $E(\eta)$ are chosen appropriately, then every deformation of a finite energy has a uniform lower bound for its Jacobian. Thus, the minimizer is well separated from the boundary of the domain in $E(\eta)$ and variations at the minimizer can be performed.
\end{remark}

Let us just remark that the second term in $E(\eta)$ assures that the stored energy blows up whenever $\det({\nabla \eta}) \to 0$ and as such penalizes compression (in fact we get a uniform energy dependent lower bound on the determinant, see \cite{healeyInjectiveWeakSolutions2009}). This is a crucial requirement non-linear elasticity (see e.g. \cite{CiarletBook3}). We also stress that such a requirement alone rules out convexity of the stored energy, albeit when $E(\eta)$ is of the form of \eqref{st-venant} also the first term is not convex.

Much more challenging for the analysis than the non-convex terms in the energy is the fact, that the admissible set of deformations has to be chosen carefully; in particular it is bound to be a non-convex set. This is related to the physical expectation that any admissible deformation will be injective.
Indeed, while it might be possible to make sense of the equations of the problem even in the case that two different parts of the solid occupy the same Eulerian space, physically such a situation is of course nonsensical. Now observe that locally the blow-up of the determinant already guarantees injectivity of any deformation of finite energy. But in addition to the local injectivity we further have to take into account the important non-convex restriction of global injectivity.
What thus turns out to be a good choice for the admissible set of $\eta$ is the following set that is in coherence with the celebrated \emph{Ciarlet-Ne\v{c}as condition} 
proposed in~\cite{CiaNec87}: 
\begin{align}
\label{eq:etaspace}
  \mathcal{E} := \left\{\eta \in W^{2,q}(Q;\Omega)\,:\, E(\eta) < \infty,\, \abs{\eta(Q)} = \int_Q \det \nabla \eta \,dx \right\}.
 \end{align}
 Here, the finite energy guarantees local injectivity and together with the last condition we have that any $C^1$-local homeomorphism is globally injective except for possible touching at the boundary. Hence in the following we will construct deformations in the set \eqref{eq:etaspace}. In fact the set $\mathcal{E}$ only includes deformations for which the dissipation function $R$ of the solid deformation, that we introduce next, is also well defined.

As for the dissipation function, we also assume an additive splitting of the fluid and solid dissipation, respectively. In addition to that we introduce a phenomenological coupling term, a \emph{drag term}, that introduces additional dissipation if the velocities of the fluid and solid are not equal. For the purposes of this paper we stick to the simplest possible setting when the drag force is proportional to the difference of the solid and fluid velocity (see \cite{rajagopal2007hierarchy}); this is also the setting within which, upon further simplifications, the law of Darcy or the equations of Brinkman are derived \cite{rajagopal2007hierarchy}. However since one of these velocities is given in Lagrangian and the other in Eulerian representation, we need to transform one of them into the proper frame, the easier of the two options being the Lagrangian. Thus, we obtain for the the dissipation functional:
\begin{align*}
\mathcal{D}= \underbrace{\frac{\nu}{2}\norm[\Omega]{\nablasym v}^2}_\text{fluid dissipation} + \underbrace{R(\eta,\partial_t \eta)}_\text{solid disspation} + \underbrace{A(\eta,\partial_t \eta- v\circ \eta)}_\text{drag potential}
\end{align*}
where 
\[
A(\eta,\phi):=\int_Q \frac{1}{2} \phi^T \cdot a(\nabla \eta) \cdot \phi \,dx
\]
 and  $a:\{M \in \R^{n\times n}: \det M > 0 \} \to \R^{n\times n}$ is a smooth function mapping into the set of positive semi-definite symmetric matrices. Note that the drag potential is quadratic, so that the drag force depends linearly, through a factor depending on $\nabla \eta$, on the difference of the two velocities. This quite general dependence on $\nabla \eta$ is necessary to model anisotropy in the correct frame independent way.\footnote{If we assume the simpler case of an isotropic solid with regards to drag, we can instead use $A(\eta,b) = \frac{a_0}{2} \norm[Q]{b}^2$, but doing so does not fundamentally change any of the later computations. Note also that for physical reasons the drag does only depend on the relative and not on the absolute velocities.} Finally, we need to prescribe the solid dissipation, for which we choose the following form which leads to an viscoelastic solid of Kelvin-Voigt type:
 \begin{align}
R(\eta,\partial_t\eta) &:= \int_Q |(\nabla \partial_t \eta)^T \nabla \eta + (\nabla \eta)^T (\nabla \partial_t \eta)|^2 dx = \int_Q |\partial_t (\nabla \eta^T \nabla \eta)|^2 dx.  \label{kelvin-voigt}
 \end{align}

Here, let us notice that the {\em dissipation of the solid} depends on the state. This is indeed necessary to assure frame indifference~\cite{antman}. 

 
With all quantities in hand, we can now write down the (formal) energy equality
\begin{align*}
 &\underbrace{E(\eta(T)) + \frac{\rho_s}{2}\norm[Q]{\partial_t \eta(T)}^2 + \frac{\rho_f}{2}\norm[\Omega]{v(T)}^2}_\text{energy at final time $T$} + \underbrace{\int_0^T \left[ \nu \norm[\Omega]{\nablasym v}^2 + 2R(\eta,\partial_t \eta) + 2A(\eta,\partial_t \eta - v \circ \eta) \right] dt}_\text{dissipated energy} \\
 &= \underbrace{E(\eta_0) + \frac{\rho_s}{2}\norm[Q]{b}^2 + \frac{\rho_f}{2}\norm[\Omega]{v_0}^2}_\text{energy at the initial time $0$} + \underbrace{\int_0^T \inner[\Omega]{f}{v} dt.}_\text{power of the external forces}
\end{align*}
Here $f$ is an external force acting on the fluid. Forces acting on the solid can be considered in a similar fashion. We use the subscript ``0'' to denote the initial data for the solid deformation and the fluid velocity given at $t=0$, while $b$ is used for the initial solid velocity.

Once the stored energy and the dissipation function have been specified, we embark on giving a weak and a strong formulation of the considered problem. This formulation needs to be closed using boundary conditions on the interface between the pure fluid and the porous media. We do so in the least invasive way, by requiring weak continuity of the fluid ({in the sense that} $v\in W^{1,2}(\Omega;\R^n)$) and stress free boundary condition for the solid deformation.

Now, we start with the weak formulation and fix the space of test functions, which consists of the pairs $(\phi,\xi) \in W^{2,q}(Q;\R^n) \times W^{1,2}_0(\Omega;\R^n)$ with $\phi|_P = 0$. Taking the derivatives of stored energy and dissipation with respect to position and velocity respectively, in directions $\phi$ and $(\phi,\xi)$ gives us the corresponding forces and we end up with:
\begin{definition}
\label{def:weak}
We call the pair $(\eta, v) \in L^\infty(0,T;\mathcal{E})\cap W^{1,2}(0,T;W^{1,2}(Q)) 
\times L^2(0,T;W^{1,2}_0(\Omega;\R^n))$ with $\eta|_P = 0$ and $\eta(0)= \eta_0$ satisfying
\begin{align} 
 \label{eq:weakSol} 
\inner[Q]{b}{\phi(0)}+\inner[\Omega]{v_0}{\xi(0)}&=\int_0^T \rho_s \inner[Q]{\partial_t \eta }{\partial_t \phi} + \inner{DE(\eta)}{\phi} + \inner{D_2R(\eta,\partial_t \eta)}{\phi} + \nu \inner[\Omega]{\nablasym v}{\nablasym \xi}\
\\ \nonumber
&\quad  + \inner{D_2A(\eta,\partial_t \eta -v \circ\eta)}{\phi-\xi \circ \eta} 
 + \rho_f \inner[\Omega]{v}{\partial_t \xi - v \cdot \nabla \xi} - \inner[\Omega]{f}{\xi} dt,
\end{align}
for all $(\phi,\xi)\in C^\infty([0,T]\times Q)\times C^\infty(0,T;C^\infty_0(\Omega))$, with $\phi(T)=0$, $\phi|_{P}=0$, $\diver \xi = 0$ and $\xi(T)=0$
a weak solution of the FPSI-problem given by \eqref{st-venant} and \eqref{kelvin-voigt} with respective initial and boudnary values.
\end{definition}
The corresponding strong formulation reads

\begin{align}
\label{strong}
 \left\{ \begin{aligned}
          \rho_s \partial_t^2 \eta &= \diver \sigma  &\text{ in } &[0,T] \times Q\\
          \diver \sigma &= DE(\eta) + D_2R(\eta,\partial_t \eta) + a(\nabla \eta) \cdot (\partial_t \eta -v\circ \eta) &\text{ in } &[0,T] \times Q\\
          \diver v &= 0 &\text{ in } &[0,T] \times \Omega\\
          \rho_f(\partial_t v + v\cdot \nabla v) &= \nu \Delta v + F -f - \nabla p& \text{ in } &[0,T] \times \Omega\\
          F&= (\det \nabla (\eta^{-1}))a(\nabla \eta \circ \eta^{-1}) \cdot (v-\partial_t \eta \circ \eta^{-1}) & \text{ in } &[0,T] \times \eta(Q)\\
          F&= 0 &\text{ in } & [0,T] \times\Omega \setminus \eta(Q)\\
          \sigma \cdot n &= 0 &\text{ in } & [0,T] \times\partial Q \setminus P
          \\
          \eta&=\gamma &\text{ in } & [0,T] \times\partial  P
          \\ 
          v&=0 &\text{ in } & [0,T] \times\partial  \Omega
          \\
          v(0)&=v_0 &\text{ in } &  \Omega
          \\
          \eta(0)=\eta_0 &\text{ and }\partial_t\eta(0)=b &\text{ in } &  Q,
         \end{aligned}
 \right.
\end{align}
where $\sigma$ is the Piola-Kirchhoff stress tensor derived from $E$ and $F$ is the Eulerian drag from the solid acting on the fluid. In deriving the strong formulation from the energy contribution, the drag term is used twice (it appears in the PDE for the velocity $v$ as well as the deformation $\eta$) but deriving both constituents from the same term guarantees that these two forces are indeed opposite and equal. Notice also that the boundary value of $\sigma$ is of Neumann type and due to a balance of stresses on the porous medium - fluid interface; in the weak formulation this is already encoded implicitly.

The goal of this paper is to prove the following theorem:

\begin{theorem}[Existence for the porous media problem]
\label{thm:main}
 Given $T> 0$,  $f \in L^2([0,T]\times\Omega;\R^n)$, $\eta_0 \in \mathcal{E}$, $b\in L^2(Q;\R^n)$, $v_0 \in L^2(\Omega;\R^n)$ there exists time $T^* \in (0,T]$ and a weak solution (in the sense of \eqref{eq:weakSol}) to \eqref{strong}. More precisely we show the existence of a pair 
 \[(\eta,v) \in C^0([0,T^*];W^{2,q}(Q;\R^n)) 
 \cap W^{1,2}([0,T^*];W^{1,2}(Q;\R^n))\times W^{1,2}([0,T^*] \times \Omega; \R^n)
 \]
  with initial data $\eta(0) = \eta_0$, $\partial_t \eta_0 = b$, $v(0) = v_0$ (in the weak sense), satisfying \eqref{eq:weakSol}. 
 Additionally this solution satisfies the physical energy inequality. Here $T^*$ either equals $T$ or is the the first time of (self-) contact, where the free part of $\partial \eta(T^*,Q)$ either touches itself or $\partial \Omega$.
\end{theorem}

Theorem \ref{thm:main} is proved in Section \ref{sec:4}. However, as the method of proof has the potential to be extended to more general models (see below) we first present, for a better understanding of the techniques, a respective overview in Section \ref{sec:2}.

\subsection*{Limitations and generalizations}

The model considered in this paper for which we prove existence of weak solutions in Theorem \ref{thm:main} includes many important features that have not been considered in mathematical FPSI yet: namely, the model is \emph{fully dynamic including inertia} and the response of the solid is \emph{fully nonlinear allowing for large structural deformations}. On the other hand, we have made several simplifying assumptions including the following ones:

\begin{enumerate} 
    \item  {\bf Non-saturated/non-zero volume porous media:} In the model at hand, the volume fraction of the fluid inside the porous medium is always fixed to one. This greatly simplifies the discussion as it in particular implies $\diver v = 0$ everywhere, even over the interface. It is however possible to generalize this condition and assume divergence free flow only on a part of the domain. In particular in \cite{BenKamSch20} we already dealt with a global velocity field that is only divergence free outside the solid.
    \item {\bf Dependence of the solid stored energy on fluid variables:} In the model we assume that the stored energy of the solid does not depend on the properties of the saturating fluid. Generalizations in this directions are known in modeling and include e.g.\ Biot's model where the stored energy depends on the pressure of the fluid in pores or models that depend on the volume fraction of the fluid and solid in each point. 
    \item {\bf Compressibility of the fluid:} The fluid is assumed incompressible and Newtonian in the present model. Nonetheless, in view of possible applications like the breathing through masks, incorporating compressibility seems an important task. This will call for changes in the analysis, but possibly also modelling, in particular if the pressure additionally enters the solid stored energy like in Biot's model. (See \cite{BreKamSch21} for an example of the variational strategy in the context of fluid-structure interaction with a compressible fluid).
    \item {\bf Drag force:} We assume the most basic form of the drag force that depends linearly on the difference between the fluid and solid velocity. It is also just the drag force where the coupling between the fluid and solid velocities happens in the porous medium model. One could think of a drag force that depends also on other variables included in the model (like volume fractions) or other coupling possibilities. 
     \item {\bf Interface conditions and boundary layers:} The interface condition at the transition between the porous medium and the circumlying fluid currently consists of continuity of the fluid velocity and a stress free boundary conditions for the solid. Generalizations are thinkable for instance an additional boundary layer described by an additional dissipation functional.
     \item {\bf Geometry of the porous medium:} In many cases the porous medium can be well described as a lower-dimensional structure, thus an appropriate membrane, shell or plate model is applicable. Applications include cell walls or the aforementioned breathing through masks. The variational approach presented here is well suited to handle also these situations; for the classical FSI this situation is currently investigated~\cite{KamSchSpe20} so that generalizations to the FPSI setting seem feasible.
     
\end{enumerate}

It seems plausible that {the discussed aspects (or other)} could be incorporated into the presented model in a future work. Let us however remark that some limitations of the given model seem not to be {removable} with the current state of art. Most significantly for the solid model. In particular it is essential to include regularizations featuring the second gradient (see Remark \ref{non-simple}).
 
 \section{Variational strategy - Overview of the proof of Theorem \ref{thm:main}}
 \label{sec:2}

We are interested in proving existence of weak solutions to the system \eqref{strong} in the sense of Definition \ref{def:weak}. As is standard, we will first prove existence of approximations of \eqref{strong} and then pass to the limit. To enable the limit passage a crucial prerequisite is the availability of suitable a-priori estimates, which are usually based on an appropriate energy (in)equality. Indeed, multiplying formally the balances of the respective momenta in \eqref{strong} by the velocity and integrating over $[0,T]$ leads to
\begin{align*}
 &\underbrace{E(\eta(T)) + \frac{\rho_s}{2}\norm[Q]{\partial_t \eta(T)}^2 + \frac{\rho_f}{2}\norm[\Omega]{v(T)}^2}_\text{energy at final time $T$} + \underbrace{\int_0^T \left[ \nu \norm[\Omega]{\nablasym v}^2 + 2R(\eta,\partial_t \eta) + 2A(\eta,\partial_t \eta - v \circ \eta) \right] dt}_\text{dissipated energy} \\
 &= \underbrace{E(\eta_0) + \frac{\rho_s}{2}\norm[Q]{\partial_t \eta(0)}^2 + \frac{\rho_f}{2}\norm[\Omega]{v_0}^2}_\text{energy at the initial time $0$} + \underbrace{\int_0^T \inner[\Omega]{f}{v} dt}_\text{power of the external forces},
\end{align*}
which, on a more abstract level, takes the compact form
 \begin{align*}
  E_{\text{st}}(T)+E_{\text{kin}}(T) + \int_0^T W_{\text{diss}}(t) dt = E_{\text{st}}(0) + E_{\text{kin}}(0) + \int_0^T W_{\text{ext}}(t)dt
 \end{align*}
 where we distinguished the the following four contributions: The stored energy $  E_{\text{st}}$, the kinetic energy $E_{\text{kin}}$, energy lost through dissipation $ W_{\text{diss}}$ and work done by external forces $W_{\text{ext}}$. The kinetic energy and the external forces will generally each always have a similar form, independent of the considered material. On the other hand, the other two carry the constitutive information on material modelling. In this work, we stick to the form \eqref{st-venant} and \eqref{kelvin-voigt} of these contributions, but the method of proof is to a large extent independent of those, so we shall not use the specific form within this section.
 
 We will prove Theorem \ref{thm:main} by \emph{time-discretization}; this is particularly suited as the geometry between the porous medium and the circumlying fluid changes during the evolution. Thus, even if we do not deal with the situation here, a time-discretization method has the potential to handle complex transition conditions between the porous medium and the fluid. More importantly, however, a time-discretization allows us to set up a \emph{variational problem} to prove \emph{existence of approximate solutions}. This is of great importance as standardly used fixed point methods do not function well if the included operators are not monotone (in some sense) and/or the set of admissible functions is not convex. 
 
 Nonetheless, the selected time-discretization method needs to be chosen carefully as the balance of momentum for the solid porous medium (i.e.\ first equation in \eqref{strong}) is \emph{hyperbolic featuring a non-monotone differential operator stemming from the non-convex stored energy}. Let us explain the difficulty on a simple toy model involving a single unit mass particle with position $x(t)\in \R^n$ and a \emph{non-convex} potential energy $E(x(t))$. Hence, we seek the solution to the following hyperbolic ODE: 
 $$\partial_t^2 x = -\nabla E(x)$$ 
 with initial data $x(0)=x_0$ and $\partial_tx(0)=x^*$. The naive ansatz is to consider a time-discretization with step-size $\tau$ and simply set up a backward Euler discretization as follows:
 \begin{align} \label{eq:naiveDiscr}
  0 = \nabla E(x_{k+1}) + \frac{\tfrac{x_{k+1}-x_k}{\tau}-\tfrac{x_k-x_{k-1}}{\tau}}{\tau}.
 \end{align}
 Assuming that existence of solutions to \eqref{eq:naiveDiscr} can be proved, one would like to derive an energy inequality from \eqref{eq:naiveDiscr} to obtain a-priori estimates. Mimicking the continuous case, the most straighforward way to do so seems to test \eqref{eq:naiveDiscr}  with the discretized time derivative. This yields
 $$
 0=\left\|\tfrac{x_{k+1}-x_k}{\tau}\right\|^2 - \inner{\tfrac{x_{k}-x_{k-1}}{\tau}}{\tfrac{x_{k+1}-x_k}{\tau}} + \inner{\nabla E(x_{k+1})}{\tfrac{x_{k+1}-x_k}{\tau}},
 $$
 where the second term can be estimated using Young's inequality. Thus, we need to concentrate on the last term; in the time-continuous case we would apply a chain rule here. However in the discrete setting a chain rule is generally not available, only if $E$ would be convex the following inequality 
 $$
 \inner{\nabla E(x_{k+1})}{\tfrac{x_{k+1}-x_k}{\tau}} \geq E(x_{k+1}) - E(x_{k}),
 $$
 that is sometimes called the \emph{discrete chain rule} could be applied. In the \emph{non-convex setting} this inequality is false in general and this path to a-priori estimates seems closed.\footnote{In some cases, e.g.\ in the case of $\lambda$-convexity, it is still possible to derive a variant of the discrete chain rule with some additional error terms on the right hand side. This can lead the way to an a-priori estimate and a resulting existence proof but is much more problem specific than the general method we present here.}
 
 The solution to this quandry has been found in \cite{BenKamSch20} by noting that rather a \emph{two-scale} approximation is needed in  \eqref{eq:naiveDiscr}. Thus we consider two parameters $\tau,h$ and write 
  \begin{equation}
  \label{eq:twoScales}
  0 = \nabla E(x_{k+1}) + \frac{\tfrac{x_{k+1}-x_k}{\tau}-\tfrac{x_k-x_{k-1}}{\tau}}{h}.
\end{equation}
 If $h$ is fixed and $t< h$, \eqref{eq:twoScales} can be viewed as a $\tau$-time discretization of the following gradient flow problem under forcing
 \begin{align}
\label{eq:approx}
 \nabla E(x(t)) = - \frac{\partial_t x(t)- x^*}{h}, \quad x(0)=x_0,
\end{align}
where $x^* = \partial_t x(0)$. Indeed, if we could pass to the limit $\tau \to 0$ in \eqref{eq:twoScales} we construct a solution $x^h$ to \eqref{eq:approx} on $[0,h]$. Now using the known values of $\partial_t x(t-h)$ in place of $x^*$ for the next interval of length $h$ we can iteratively prolong this process to get what we will call a \emph{time-delayed solution}, satisfying
\begin{align}
\label{eq:approx1}
 \nabla E(x(t)) = - \frac{\partial_t x(t)- \partial_t x(t-h)}{h}
\end{align}
for $t \in [0,T]$. 
We can test this time-delayed solution with $\partial_t x(t)$ and find the {\em hyperbolic a-priori estimate} 
\begin{align*}
 E(x(b)) - E(x(a)) &= - \int_a^b\inner{\frac{\partial_t x(t)-\partial_t x(t-h)}{h}}{\partial_t x(t)}\, dt 
 \\
 &\quad \leq -\frac{1}{2} \fint_{b-h}^{b} \abs{\partial_t x(t)}^2 dt +\frac{1}{2} \fint_{a-h}^a \abs{\partial_t x(t)}^2 dt,
\end{align*}
whenever the solution was constructed over $[a-h,b]$.\footnote{Actually in the construction procedure the hyperbolic a-priori estimate already needs to be used for each interval of length $h$ to guarantee the admissibility of the initial/right hand data of the next.} 

The above gives us a good estimate on $E(x(t))$ and an averaged time-derivative,
independent of $h$ which allows for sending $h \to 0$ in~\eqref{eq:approx1}. Thus this idea paves the path to Theorem \ref{thm:main}.
 
On the level of the $\tau$-discretization we deal with the gradient flow problem \eqref{eq:approx1}. For these variational methods are well established. We will explain this in the following. Let us consider the rescaled problem with $h=1$ so that we obtain
 $$
 \nabla E(x(t)) = - \partial_t x(t)-f, \quad x(0)=x_0,
 $$
 with some given time dependent force $f$, which includes the previously constructed $\partial_t x(t-h)=\partial_t x(t-1)$. We may discretize this by means of the backward Euler method to get  
\begin{align}
\label{eq:parab:disc}
     0= \nabla E(x_{k+1}) + \frac{x_{k+1}-x_k}{\tau} + f_k.
\end{align}
where $f_k$ is a corresponding time discretization of $f$.
 Now one realizes that \eqref{eq:parab:disc} is actually the Euler-Lagrange equation to the minimization problem 
 $$
 \mathcal{F}_k(x) \longrightarrow \text{min},
 $$
 where $\mathcal{F}_k: x\mapsto E(x) + \frac{1}{2\tau}\abs{x-x_k}^2+ (x-x_k) \cdot f_k$. Thus, one constructs $x_{k+1}$ using the minimization problem. Now, since one specifically has a minimizer, instead of using the equation one can compare the values of $F_k(x)$ at $x_{k+1}$ and $x_k$ in order to derive a-priori estimates. This leads to
 \begin{align*}
  E(x_{k+1}) + \frac{\tau}{2}\abs{\frac{x_{k+1}-x_k}{\tau}}^2 + \tau\frac{x_{k+1}-x_k}{\tau} f= \mathcal{F}_k(x_{k+1}) \leq \mathcal{F}_k(x_k) = E(x_k),
 \end{align*}
 and summing up yields the energy inequality (up to a factor on the dissipation) and, in turn, a-priori estimates. This method, when one resorts to minimization, has been known as the \emph{minimizing movements approximation} of gradient flows and is due to De Giorgi \cite{DeGiorgi}. Since then it has been widely used in mathematical solid mechanics since it allows to cope with the non-convexities of $E$ (see e.g. \cite{KruRou19}) and is well suitable also in our case. 

\begin{remark}[Minimization in the hyperbolic case]
A more direct attempt at a proof of the main theorem would be to generalize the minimizing movements by minimizing the functional
 \begin{align} \label{eq:naiveMin}
 \mathcal{F}_k(x) := E(x) + \frac{1}{2} \abs{\tfrac{x-x_k}{\tau}-\tfrac{x_k-x_{k-1}}{\tau}}^2,
 \end{align}
 which is easily checked to have \eqref{eq:naiveDiscr} as Euler-Lagrange equation. But then the same arguments apply to establishing estimates via the equation and comparing the values of $F_k(x)$ at $x_{k+1}$ and $x_k$ instead yields
 \begin{align*}
  E(x_{k+1}) + \frac{\tau^2}{2}\abs{\frac{\tfrac{x_{k+1}-x_k}{\tau}-\tfrac{x_k-x_{k-1}}{\tau}}{\tau}}^2 = \mathcal{F}_k(x_{k+1}) \leq \mathcal{F}_k(x_k) = E(x_k) + \frac{1}{2} \abs{\tfrac{x_k-x_{k-1}}{\tau}}^2
 \end{align*}
 with a term on the right hand side that turns out to have entirely the wrong scaling to estimate.\footnote{This is not surprising, as we are comparing a proper approximately inertial solution with one that suddenly stops. A better competitor might be the ``straight continuation'' $x_k+ \tau(x_k-x_{k+1})$, but then the estimate again requires convexity to deal with the energy-term.} Thus, the energy estimate cannot be obtained.
 \end{remark}

The approach explained here turns out to be admissible for infinite-dimensional spaces instead of $\R^n$ and even coupling between Eulerian and Lagrangian coordinates. For that Eulerian-Lagrangian coupling however additional difficulties appear that require some novel ideas on its own. This will be discussed in the forthcoming sections.

\begin{remark}[Numerical use of the method]
\label{rem:num}
Since numerous numerical schemes for minimization (over discrete spaces) are available the above methodology might also be attractive for computational mathematics. The idea would here be to do a two-scale approximation:
This means that once $x_{k+1}^{\ell-1},x_{k}^{\ell-1}$ and $x^\ell_{k}$ are constructed. Then we can define $x_{k+1}^{\ell}$ as the minimizer of
 \begin{align*}
  \mathcal{F}_k^{\ell}(x) := E(x) + \frac{\tau}{2h} \abs{\frac{x-x_k^{\ell}}{\tau} - \frac{x_{k+1}^{\ell-1}-x_k^{\ell-1}}{\tau}  }^2.
 \end{align*}
 In order to pass to the limit it is in general unavoidable to use the hyperbolic structure on a time-continuous level.  This means that first $\tau\to 0$ and only afterwards $h\to 0$. The question is {\em how much smaller} does $\tau$ needs to be? One observes quickly, that in case $E$ is convex $\tau$ and $h$ can be chosen arbitrarily. Hence, the smallness of $\tau$ in relation should depend on the {\em non-convexity} of the assumed energies. In a forthcoming paper we hope to investigate this issue further.
\end{remark}

\section{Proof of Theorem \ref{thm:main}}
\label{sec:4}
\noindent
In this section we prove Theorem \ref{thm:main}, i.e. we prove existence of weak solutions of the FPSI-problem \eqref{strong}. We follow the general strategy as outlined in Section \ref{sec:2} and as devised in \cite{BenKamSch20}. First we show existence for solutions to a time-delayed, parabolic problem and then we use these solutions to construct the weak solution to the actual problem. These two parts will be the topic of the next two subsections. 

 \subsection{Existence for the time-delayed problem}
 
 The key to finding a solution to the time-delayed problem is to understand that on short time-scales, the time-delayed problem is ultimatively parabolic. Instead of treating the delayed velocities $\partial_t \eta(t-h)$ and $v(t-h)$ as non-local (in time) parts of the equation, we can treat them as a fixed given data by considering times $t<h$ only. This mindset allows us to state and prove the following Proposition \ref{prop:timeDelayed} during this subsection.  
 
 Additionally, in this subsection  and in Proposition \ref{prop:timeDelayed}, we need to introduce regularizations of the involved stored energy and dissipation function depending on the following parameters: We take $k_0 \in \N$ in such a way that $W^{k_0,2}(\Omega;\R^n)$ embeds into $C^1(\Omega;\R^n)\cap W^{2,q}(Q;\R^n)$. Further we choose $a_0 \in(0,1)$ appropriately to be fixed later.
 
 \begin{proposition}[Existence for short time time-delayed solutions] \label{prop:timeDelayed}
  Given $h > 0$ sufficiently small,  $f,w \in L^2([0,h]\times\Omega;\R^n)$, $\zeta \in L^2([0,h]\times Q;\R^n)$, $\eta_0 \in \mathcal{E}$ there exist 
  \[
  \eta \in C^0([0,h];\overline{\mathcal{E}}) \cap W^{1,2}([0,h];W^{k_0,2}(Q;\R^n))\text{  with }\eta(t)|_P = \gamma
  \]
  and $v \in W_0^{k_0,2}([0,h] \times \Omega;\R^n)$ with $\diver v = 0$ such that
  \begin{align*}
   &\int_0^h \langle DE(\eta), \phi \rangle + h^{a_0} \inner[Q]{\nabla^{k_0} \eta}{\nabla^{k_0} \phi}+ \langle D_2R(\eta,\partial_t \eta), \phi \rangle +h \inner[Q]{\nabla^{k_0} \partial_t \eta}{\nabla^{k_0} \phi} \\
   &+ \langle D_2A(\eta,\partial_t \eta -v \circ\eta),\phi-\xi \circ \eta\rangle + \nu \inner[\Omega]{\nablasym v}{\nablasym \xi} + h \inner[\Omega]{\nabla^{k_0} v}{\nabla^{k_0} \xi} \\
   &+\rho_s\inner[Q]{\tfrac{\partial_t \eta - \zeta}{h}}{\phi} + \rho_f\inner[\Omega]{\tfrac{v \circ \Phi - w}{h}}{\xi \circ \Phi} - \inner[\Omega]{f}{\xi} dt = 0
  \end{align*}
   for all $\phi \in L^2([0,h];W^{2,q}(Q;\R^n))$ with $\phi_P =0$ and $\xi \in L^2([0,h];W^{1,2}_0(\Omega;\R^n))$ with $\diver \xi = 0$. Here $\Phi:[0,h]\times \Omega \to \Omega$ is the flow map of $v$, i.e.\ a family of volume preserving diffeomorphisms $\Phi(t,.):\Omega\to \Omega$ such that $\Phi(0,.) = \textrm{id}$ and $\partial_t \Phi(t,y) = v(t,\Phi(t,y))$.
 \end{proposition}
In accordance with the regularizing parameters, we define 
 \begin{align*}
  E_h(\eta) := E(\eta) + \frac{h^{a_0}}{2} \norm[Q]{\nabla^{k_0} \eta}^2 \text{ and }   R_h(\eta,b) := R(\eta,b) + \frac{h}{2} \norm[Q]{\nabla^{k_0} b}^2
 \end{align*}
 as a shorthand and note that
 \begin{align*}
  \inner{DE_h(\eta)}{ \phi } = \inner{DE(\eta)}{ \phi } + h^{a_0} \inner[Q]{\nabla^{k_0} \eta}{\nabla^{k_0} \phi}
 \end{align*}
 and similarly for $\inner{D_2R_h(\eta,\partial_t \eta)}{ \phi }$. These regularizations will help us at several points, both with establishing the energy inequality as well as the flow map of the fluid.

 As explained previously, the key to deal with such problems in a manner consistent with our energy-considerations, is De Giorgi's method of minimizing movements \cite{DeGiorgi}. We thus start the proof by minimizing the functional
 \begin{align*}
  \mathcal{F}_k^{(\tau)}(\eta,v) &:=  E_h(\eta) + \tau \left[ R_h(\eta_k^{(\tau)},\partial_k^\tau \eta) +  A(\eta_k^{(\tau)},\partial_k^\tau \eta - v \circ \eta_k^{(\tau)}) + \frac{\nu}{2} \norm[\Omega]{\nablasym v}^2 +\frac{h}{2} \norm[\Omega]{\nabla^{k_0} v}^2\right]\\ &+ \frac{\tau}{2h} \left[\rho_s \norm[Q]{\partial_k^\tau \eta - \zeta_k^{(\tau)}}^2 + \rho_f\norm[\Omega]{v \circ \Phi_k^{(\tau)} - w_k^{(\tau)}}^2\right] - \tau \inner[\Omega]{f_k^{(\tau)}}{v}
 \end{align*}
 in the class of all admissible $\eta,v$, i.e.\ those such that $\eta \in \mathcal{E}_h = \mathcal{E} \cap W^{k_0,2}(Q; \R^n)$ and $v \in W_{0,\mathrm{div}}^{1,2}(\Omega;\R^n)$, i.e. $v \in W_{0}^{1,2}(\Omega;\R^n)$ with $\diver v = 0$. Here $\partial_k^\tau \eta := \smash{\tfrac{\eta-\eta_k^{(\tau)}}{\tau}}$ is a discrete derivative and $\smash{f_k^{(\tau)}} := \smash{\fint_{k\tau}^{(k+1)\tau} f dt}$ is a time-discretization of $f$. In the same way we define $\smash{\zeta_k^{(\tau)}}$ and $\smash{w_k^{(\tau)}}$. For any given $\eta_k^{(\tau)},\Phi_k^{(\tau)}$, this minimizer will yield the next step, i.e.\ the pair $(\eta_{k+1}^{(\tau)},v_{k+1}^{(\tau)})$. From this we then also construct $\Phi_{k+1}^{(\tau)}:= (\textrm{id} + \tau v_{k+1}^{(\tau)}) \circ \Phi_k^{(\tau)}$.\footnote{Note that, as expected for a parabolic problem, the rate variables, i.e.\ the velocities $\partial_k^\tau \eta^{(\tau)}_k$ and $v_k^{(\tau)}$ are almost entirely discarded. Only the latter occurs indirectly in $\Phi_k^{(\tau)}$. Similarly we are not using the initial data for these velocities. Only later will they all reappear through proper choice of $\zeta$ and $w$. This again emphasizes that these rate variables are associated with the $h$-scale which is kept constant for the time-delayed problem.}

 The existence of a minimizer to $\mathcal{F}_k^{(\tau)}$ follows by the direct method from the calculus of variations. Indeed, by the coercivity of $\smash{\mathcal{F}_k^{(\tau)}}$, we find a bounded sequence $(\eta_k^l, v_k^l) \in \mathcal{E}_h\times \smash{W^{1,2}_{0,\mathrm{div}}(\Omega;\R^n)}$, such that
 $$
 \mathcal{F}_k^{(\tau)}(\eta_k^l, v_k^l) \to \mathrm{inf}_{(\nu,\eta) \in \mathcal{E}_h\times W^{1,2}_{0,\mathrm{div}}(\Omega;\R^n)}  \mathcal{F}_k^{(\tau)}(\eta, v)
 $$
 Here, the boundedness of sequence is particularly due to the added regularizing terms, though existence of minimizers could be shown even if they were omitted, as well as the highest order term in the elastic stored energy.
 
 By the Banach-Alaoglu theorem, we conclude the existence of a pair $(\eta_k^\tau, v_k^\tau) \in W^{k_0,2}(Q; \R^n) \cap W^{2,q}(Q; \R^n) \times W^{1,2}(\Omega,\R^n)$ such that
 $$
 (\eta_k^l, v_k^l) \rightharpoonup (\eta_k^\tau, v_k^\tau) \in W^{k_0,2}(Q; \R^n)W^{2,q}(Q; \R^n) \times W^{1,2}(\Omega,\R^n).
 $$
 By compact embeddings, owing to the regularizing terms, we can show that $\eta_k^\tau$ fulfills the Ciarlet-Ne\v{c}as condition and thus lies in $\mathcal{E}_h$ \footnote{This can be show even under a weaker type of convergence; see the original paper \cite{CiaNec87}.}, due to linearity $v_k^\tau$ is divergence free and satisfies the zero boundary condition. 
 
 It remains to show that $\mathcal{F}_k^{\tau}$ is weakly lower semicontinuous with respect to the weak convergence above. This is nevertheless standard, due to the convexity of the highest order terms, the strong convergence of $\eta_k^l \to \eta_k^\tau$ in $C^1(Q;\R^n)$. Thus, $(\eta_k^\tau, v_k^\tau)$ is the sought minimizer.
 
 The reason we prefer a minimization to other approaches is that it allows us to immediately access a discrete energy inequality, without having to rely on additional properties of the energy, such as convexity.

 For this we compare the minimal value, i.e.\ that of $(\eta_{k+1}^{(\tau)},v_{k+1}^{(\tau)})$ to that of ``standing still'', i.e.\ $(\eta_k^{(\tau)},0)$ in the same functional. This removes all dissipative terms on one side of the inequality and gives us
 \begin{align*}
 &\phantom{{}={}}  E_h(\eta_{k+1}^{(\tau)}) + \tau \left[ R_h(\eta_k^{(\tau)},\partial_k^\tau \eta_{k+1}^{(\tau)}) + A(\eta_k^{(\tau)},\partial_k^\tau \eta_{k+1}^{(\tau)} - v_{k+1}^{(\tau)} \circ \eta_k^{(\tau)}) + \frac{\nu}{2} \norm[\Omega]{\nablasym v_{k+1}^{(\tau)}}^2 +\frac{h}{2} \norm[\Omega]{\nabla^{k_0} v^{(\tau)}_{k+1}}^2\right]\\ &+ \frac{\tau}{2h} \left[\rho_s\norm[Q]{\partial^\tau_k \eta_{k+1}^{(\tau)} - \zeta_k^{(\tau)}}^2 + \rho_f\norm[\Omega]{v_{k+1}^{(\tau)} \circ \Phi_k^{(\tau)} - w_k^{(\tau)}}^2\right] - \tau \inner[\Omega]{f_k^{(\tau)}}{v_{k+1}^{(\tau)}} \\
  &\leq E_h(\eta_k^{(\tau)}) + \frac{\tau}{2h} \left[\rho_s\norm[Q]{\zeta_k^{(\tau)}}^2 + \rho_f\norm[\Omega]{ w_k^{(\tau)}}^2\right]
 \end{align*}
 A telescope argument and the usual weighted Young's inequality then gives us an energy estimate as well.
 \begin{align*}
 &  E_h(\eta_N^{(\tau)}) + \! \sum_{k=0}^{N-1} \tau \left[ R_h(\eta_k^{(\tau)},\partial^\tau_k \eta_{k+1}^{(\tau)}) + A(\eta_k^{(\tau)},\partial_k^\tau \eta_{k+1}^{(\tau)} - v_{k+1}^{(\tau)} \circ \eta_k^{(\tau)})  + \frac{\nu}{2} \norm[\Omega]{\nablasym v_{k+1}^{(\tau)}}^2 \!+\frac{h}{2} \norm[\Omega]{\nabla^{k_0} v^{(\tau)}_{k+1}}^2 \right]\\ &+ \sum_{k=0}^{N-1} C \tau \left[\rho_s\norm[Q]{\partial^\tau_k \eta_{k+1}^{(\tau)}}^2 + \rho_f\norm[\Omega]{v_{k+1}^{(\tau)} \circ \Phi_k^{(\tau)}}^2\right] \\
 & \leq E_h(\eta_k^{(\tau)}) + C\sum_{k=0}^{n-1}\frac{\tau}{2h} \left[\rho_s\norm[Q]{\smash{\zeta_k^{(\tau)}}}^2 + \rho_f\norm[\Omega]{\smash{ w_k^{(\tau)}}}^2+ \norm[\Omega]{\smash{f_k^{(\tau)}}}^2\right]
 \end{align*}
This is the first of many similar energy estimates related to the physical energy inequality, which we will make use of. 

Note also that in each step this also allows us to derive a uniform estimate on the distance from the initial data in the form 
\begin{align*}
 \norm[Q]{\eta_0-\smash{\eta_{N}^{(\tau)}}} \leq \sum_{k=0}^{N-1} \tau \norm[Q]{\partial^\tau_k \eta^{(\tau)}_{k+1}} \leq \sqrt{ \sum_{k=0}^{N-1} \tau}  \sqrt{\sum_{k=0}^{N-1} \tau \norm[Q]{\partial^\tau_k \eta^{(\tau)}_{k+1}}^2} \leq  C\sqrt{h}.
\end{align*}

Since a straightforward calculation (see \cite[Prop. 2.7]{BenKamSch20})
shows that any given state of finite energy has a minimum $L^2$-distance to any (self-)coliding state of finite energy, this also tells us that we can indeed choose $h$ small enough to avoid such a collision.

Next, we define the follwoing piecewise constant and piecewise affine (in time) approximations:
 \begin{align*}
  \bar{\eta}^{(\tau)} (t,x) &:= \eta_{k+1}^{(\tau)}(x) &\text{ for }& t\in [\tau k,\tau (k+1)), x\in Q \\
  \underline{\eta}^{(\tau)} (t,x) &:= \eta_k^{(\tau)}(x) &\text{ for }& t\in [\tau k,\tau (k+1)), x\in Q \\
  \hat{\eta}^{(\tau)}(t,x) &:= ((k+1)-t/\tau)\eta_k^{(\tau)}(x) - (t/\tau-k)\eta_{k+1}^{(\tau)}(x) &\text{ for }& t\in [\tau k,\tau (k+1)), x\in Q \\
  v^{(\tau)}(t,y) &:= v_k^{(\tau)}(y) &\text{ for }& t\in [\tau k,\tau (k+1)), y\in \Omega \\
  \Phi^{(\tau)}(t,y) &:= \Phi_k^{(\tau)}(y) &\text{ for }& t\in [\tau k,\tau (k+1)), y\in \Omega \\
  \intertext{ in particular we have }
  \partial_t \hat{\eta}^{(\tau)}(t,x) &= \partial^{\tau}_k \eta_{k+1}^{(\tau)}(x) &\text{ for }& t\in (\tau k,\tau (k+1)), x\in \Omega 
 \end{align*}
 Plugging these quantities into the energy estimate above, we get that there is a constant $C$ independent of $\tau$ such that
 \begin{align*}
  \sup_{t\in [0,h]} E_h(\bar{\eta}^{(\tau)}(t)) \leq C, &\quad   \sup_{t\in [0,h]} \norm[W^{k_0,2}(Q)]{\bar{\eta}^{(\tau)}(t)} \leq C, & \quad \sup_{t\in [0,h]} \norm[W^{k_0,2}(Q)]{\hat{\eta}^{(\tau)}(t)} \leq C\\
  \int_0^h \norm[W^{k_0,2}(Q)]{\partial_t \hat{\eta}^{(\tau)}(t)}^2 dt \leq C&\quad   \int_0^h \norm[W^{k_0,2}(\Omega)]{v^{(\tau)}}^2 dt \leq C, & \int_0^h A(\eta,v^{(\tau)}\circ \bar{\eta}^{(\tau)} - \partial_t \hat{\eta}^{(\tau)}) dt \leq C.
 \end{align*}
Thus we can find a (non-relabelled) subsequence of $\tau$'s  and $\eta \in W^{1,2}([0,T];W^{k_0,2}(Q;\R^n))$ and $v \in L^2([0,T];W^{k_0,2}(\Omega;\R^n))$ such that
\begin{align*}
&\bar{\eta}^{(\tau)}, \underline{\eta}^{(\tau)} \stackrel{*}{\rightharpoonup} \eta \text{ in }  L^\infty([0,T];W^{k_0,2}(Q;\R^n)) \quad \text{ and } \quad \hat{\eta}^{(\tau)} \rightharpoonup \eta \text{ in }  W^{1,2}([0,T];W^{k_0,2}(Q;\R^n))  \\ &v^{(\tau)} \rightharpoonup v \text{ in } L^2([0,T];W^{k_0,2}(\Omega;\R^n)).
\end{align*}
 
 The next question we are going to have to deal with is how to construct the flow map $\Phi$. In each step, the flow map is updated by concatenation. Not only is this in itself an inherently non-linear procedure, but also for any non-zero time, the number of those steps goes to infinity when we send $\tau$ to zero. On the other hand, each of these steps represents a change of scale $\tau$. This hints at an exponential structure and this is indeed what we find. As often (e.g.\ when working with conservation laws) the Lipschitz continuity of the flow-velocity is essential to have a well defined Lagrangian flow map. Here we particularly rely on the regularization of the flow-velocity is used at the $\tau$ level.
 
 \begin{lemma}[Existence of a flow map]
 Assume that $v^{(\tau)}\in L^2([0,h]; C^{0,1}(\Omega;\R^n))$ uniformly in $\tau$, such that $\diver v^{(\tau)}=0$. Then 
  the maps $\Phi^{(\tau)}$ are uniformly Lipschitz continuous in space, independent of $\tau$ (but not of $h$). Additionally we have in the limit that
  \begin{align*}
   \det \nabla \Phi^{(\tau)} \to 1.
  \end{align*}
 \end{lemma}

 \begin{proof}
 Let us illustrate the proof for the Jacobian determinant. Here we are using an estimate involving an expansion of the determinant, the inequality between arithmetric and geometric mean, as well as the fact that $(1+a/N)^N \nearrow \exp(a)$:
\begin{align*}
 &\phantom{{}={}} \det \nabla \Phi_N^{(\tau)} = \prod_{k=1}^N \left[\det\left(I + \tau \nabla v_k^{(\tau)} \right)\right] \circ \Phi_{k-1}^{(\tau)} \\
 &= \prod_{k=1}^N \bigg[1 + \tau \underbrace{\tr\left(\nabla v_k^{(\tau)} \right)}_{=\diver v_k^{(\tau)} = 0}+ \sum_{l=2}^n \tau^l M_l\left(\nabla v_k^{(\tau)}\right) \bigg] \circ \Phi_{k-1}^{(\tau)} \\
 &\leq \prod_{k=1}^N \left[ 1 + \sum_{l=2}^n c \tau^l \Lip\left(\nabla v_k^{(\tau)}\right)^l\right] \leq \left( 1 + \frac{1}{N}\sum_{k=1}^N \sum_{l=2}^n \tau^l c \Lip\left( v_k^{(\tau)}\right)^l \right)^N \\
 &\leq \exp\left( C \tau \sum_{k=1}^N \tau \Lip\left( v_k^{(\tau)}\right)^2 \right), 
\end{align*} 
where $M_l$ includes all terms of order $l$ in the polynomial expansion of the determinant (e.g. $M_1(A) =\tr(A)$, ..., $M_{n-1}(A) = \tr(\cof(A))$, $M_n(A) =\det A$, cf. e.g. \cite[Lemma A.1.]{BenKamSch20}). Now the sum in the last term is nothing but $\tau$ times the $L^2([0,h]; C^{0,1}(\Omega;\R^n))$-norm of $v^{(\tau)}$, something on which we have $\tau$-independent bounds by assumption. Similar arguments also derive a bound from below, which gives us $\det \nabla \Phi_N^{(\tau)} \to 1$ uniformly. Furthermore applying a very similar argument to $\abs{\smash{\nabla \Phi^{(\tau)}_N}}$ gives us a $\tau$-independent bound on the Lipschitz constant of $\Phi$.
\end{proof}

In particular, since $\Phi^{(\tau)}$ is piecewise constant, this allows us to use a variant of the Arzela-Ascoli theorem to find a limit $\Phi:[0,h] \times \Omega \to \Omega$ (after a subsequence), which has to be a volume preserving diffeomorphism with
\begin{align*}
 \partial_t \Phi(t,y) = v(t,\Phi(t,y)) \text{ for all } t \in [0,h], y \in \Omega
\end{align*}
i.e.\ a flow map.

With this, we have one half of the proof in hand, namely constructing the objects that are going to form our solution. The other half is showing that these objects actually are a solution, i.e.\ that they satisfy the right equation.

For this we can follow along the same steps as the construction. Apart from giving us an energy inequality, the other benefit of starting with a minimization is that any such minimizer has to satisfy the Euler-Lagrange equation.

\begin{align*}
  &\inner{DE_h(\eta_{k+1})}{\phi} +  \inner{D_2 R_h(\eta_k,\partial^\tau \eta_{k+1})}{\phi} + \inner{D_2A(\eta_k,\partial^\tau \eta_{k+1} - v_{k+1} \circ \eta_k)}{\phi - \xi \circ \eta_k} \\
  &\quad+ \nu \inner[\Omega]{\varepsilon v_{k+1}}{\varepsilon \xi} +h \inner[\Omega]{\nabla^{k_0} v_{k+1}}{\nabla^{k_0} \xi}+ \rho_s\inner[Q]{\tfrac{\partial^\tau \eta-\zeta_k}{h}}{\phi} + \rho_f\inner[\Omega]{\tfrac{v_{k+1} \circ \Phi_k -w_k}{h}}{\xi} + \inner[\Omega]{f}{\xi} = 0
\end{align*}
for any pair $(\phi,\xi)$ such that $(\eta_{k+1}+\varepsilon\phi,v_{k+1}+\varepsilon \xi/\tau)$ is also admissible for our minimization\footnote{Note that we scale $\phi$ and $\xi$ differently in $\tau$, so that they both behave like a change of position. This is not required at this point, but convenient to do for the limit passage.} for some $\varepsilon_0$ and any $\varepsilon \in [-\varepsilon_0,\varepsilon_0]$\seb{,} 
$\phi \in W^{1,2}(Q;\R^n)$, $\xi\in W_0^{1,2}(\Omega;\R^n)$ with $\phi|_P = 0$ and $\diver \xi = 0$. Here, we point out that, as we start with a configuration without contact we know due to the bounds in the energy estimate that, for $h$ small enough, every minimizer will be contact-free. Thus, indeed variations in all directions are possible. 

Next we replace all occurences of the discrete approximations with the corresponding time dependent functions, giving us
\begin{align*}
  &\inner{DE_h(\bar{\eta}^{(\tau)})}{\phi} +  \inner{D_2 R_h(\underline{\eta}^{(\tau)},\partial_t \hat{\eta}^{(\tau)})}{\phi} + \inner{D_2A(\underline{\eta}^{(\tau)},\partial_t \hat{\eta}^{(\tau)} - v^{(\tau)} \circ \underline{\eta}^{(\tau)})}{\phi - \xi \circ \underline{\eta}^{(\tau)}}  \\
  &\quad+ \nu \inner[\Omega]{\varepsilon v^{(\tau)}}{\varepsilon \xi}+h \inner[\Omega]{\nabla^{k_0} v^{(\tau)}}{\nabla^{k_0} \xi} \\
  &\quad+ \rho_s\inner[Q]{\tfrac{\partial_t \hat{\eta}^{(\tau)}-\zeta_k}{h}}{\phi} + \rho_f\inner[\Omega]{\tfrac{v^{(\tau)} \circ \Phi^{(\tau)} -w_k}{h}}{\xi\circ \Phi^{(\tau)}} + \inner[\Omega]{f}{\xi} = 0
\end{align*}
for almost all times $t$.

Note in particular that there is now no longer any direct occurrence of $\tau$ in this equation. Thus one can form the integral over $[0,h]$ and pass to the limit $\tau \to 0$. Due to the added regularizing terms this is straightforward. Indeed, the highest terms (added through linearization) are linear. For all other terms (except the $\frac{1}{\mathrm{det} (\nabla \eta)^a}$ for which we rely on uniform continuity) we exploit compact embeddings and Nemitskii continuity (see e.g. \cite{roubcek.t:nonlinear}) and thus obtain:
\begin{align*}
  &\int_0^h \inner{DE_h(\eta)}{\phi} +  \inner{D_2 R_h(\eta,\partial_t \eta)}{\phi} + \inner{D_2A(\eta,\partial_t \eta - v \circ \eta)}{\phi - \xi \circ \eta}  \\
  &\quad+ \nu \inner[\Omega]{\varepsilon v}{\varepsilon \xi}+h \inner[\Omega]{\nabla^{k_0} v}{\nabla^{k_0} \xi}+ \rho_s\inner[Q]{\tfrac{\partial_t \eta-\zeta_k}{h}}{\phi} + \rho_f\inner[\Omega]{\tfrac{v \circ \Phi -w}{h}}{\xi\circ \Phi} + \inner[\Omega]{f}{\xi} dt = 0
\end{align*}
Finally we need to show that the new solution also satisfies the energy inequality. There is a way to do so directly from the approximation using De Giorgi's methods.\footnote{Note that our previous energy estimate was off by a factor 2 in all dissipation terms, so we cannot simply apply lower-semicontinuity to this.} However in the situation that we are studying, due to the regularizer, we can directly show that the energy inequality holds true for all solutions.

For this, we test the equation directly with the pair $(\partial_t \eta,v)$, which we are allowed due to our regularizing terms and we obtain

\begin{align*}
& E_h(\eta(h)) + \int_0^h \inner{D_2 R_h(\eta,\partial_t \eta)}{\partial_t \eta} + 2A(\eta,\partial_t \eta - v \circ \eta) + \nu \norm[\Omega]{\varepsilon v}^2 + h \norm[\Omega]{\nabla^{k_0} v}^2 dt \\
  &\quad+  \fint_0^h \rho_s \inner[Q]{\partial_t \eta-\zeta}{\partial_t \eta} + \rho_f\inner[\Omega]{v \circ \Phi -w}{v \circ \Phi} dt + \int_0^h \inner[\Omega]{f}{v} dt = E_h(\eta(0))
\end{align*}

Now we apply a consequence of Young's inequality, $\inner{a-b}{a} = \abs{a}^2 - \inner{b}{a} \geq \frac{1}{2} \abs{a}^2 - \frac{1}{2} \abs{b}^2$, to the two inertial terms. Together with the conservation of volume implying $\norm[\Omega]{v \circ \Phi}^2 = \norm[\Omega]{v}^2$, this results in the energy inequality:

\begin{align*}
& E_h(\eta(h)) + \int_0^h 2R_h(\eta,\partial_t \eta) + 2A(\eta,\partial_t \eta - v \circ \eta) + \nu \norm[\Omega]{\varepsilon v}^2 +h \norm[\Omega]{\nabla^{k_0} v}^2 dt 
\\&+  \fint_0^h \frac{\rho_s}{2} \norm[Q]{\partial_t \eta}^2 + \frac{\rho_f}{2}\norm[\Omega]{v  } dt  = E_h(\eta(0)) + \fint_0^h \frac{\rho_s}{2} \norm[Q]{\zeta}^2 + \frac{\rho_f}{2}\norm[\Omega]{w} dt +  \int_0^h \inner[\Omega]{f}{v} dt.
\end{align*}

\subsection{Convergence to the full problem}
 
Taking a look at the solutions constructed in the previous section, we can use the energy inequality to see that a solution on the interval $[0,h]$ can be turned into admissible initial and right hand side data with $\eta(h)$ as the new $\eta_0$, $\partial_t \eta$ as $\zeta$, $v(t) \circ \Phi(t) \circ \Phi(h)^{-1}$ as $w$ (to correct for the flow) and so on. So we can apply the result again to construct a solution on the interval $[h,2h]$ and so on. We do so and combine these solutions into a single pair of functions $\eta^{(h)}: [0,T] \times Q \to \Omega$ and $v^{(h)}: [0,T] \times \Omega \to \R^n$. These then fulfill the full time-delayed equation:

\begin{align} \label{eq:longTimeDelayed}
\begin{aligned}
  &\int_0^T \rho_s\inner[Q]{\tfrac{\partial_t \eta^{(h)}(t)-\eta^{(h)}(t-h)}{h}}{\phi} + \rho_f\inner[\Omega]{\tfrac{v^{(h)}(t) -v^{(h)}(t-h)\circ \Phi_{-h}^{(h)}(t) }{h}}{\xi}\, dt
  \\
 &\quad + \int_0^T \inner[Q]{DE_h(\eta^{(h)})}{\phi}   + \inner[Q]{D_2A(\eta,\partial_t \eta^{(h)} - v^{(h)} \circ \eta^{(h)})}{\phi - \xi \circ \eta^{(h)}}
  \\
  &\quad  + \int_0^T\inner[Q]{D_2 R_h(\eta^{(h)},\partial_t \eta^{(h)})}{\phi}+\nu \inner[\Omega]{\varepsilon v^{(h)}}{\varepsilon \xi} + h\inner{\nabla^{k_0} v^{(h)}}{\nabla^{k_0}\xi}\, dt = \int_0^T\inner[\Omega]{f}{\xi} dt
  \end{aligned}
\end{align}
Special care needs to be taken in constructing all the fluid related terms along the way. For the flow map we take $\Phi_s^{(h)}(t,.)$ to be the map that maps a fluid particle's position at time $t$ to its position at time $t+s$. Such a map is constructed from composing the flow maps on each of the subintervals, by first using their inverse to get to a multiple of $h$ and then moving forward again. 

Or in other words, if $\Phi^0,\Phi^1,\dots$ denote the flow maps on the respective intervals $[0,h], [h,2h],\dots$ then we first define for $t \in [lh,(l+1)h]$
\begin{align*}
 \Phi_t^{(h)}(0) := \Phi^l(t-lh) \circ \Phi^{l-1}(h) \circ \dots \circ \Phi^0(h)
\end{align*}
and then for arbitrary $t,t+s \in [0,T]$ using the invertibility of the flow maps
\begin{align*}
 \Phi^{(h)}_s(t) := \Phi^{(h)}_{t+s}(0) \circ \Phi^{(h)}(t)^{-1}.
\end{align*}
In practice of course we are mostly interested in the case when $t$ and $t+s$ are at most a distance $h$ apart and most of the back and forth in this definition cancels.

In particular these maps are flow maps for $v^{h}$, so we have
\begin{align*}
 \partial_s \Phi_s^{(h)}(t,y) = v^{(h)}(t+s,\Phi_s^{(h)}(t,y)) \text{ for all } t,t+s \in [0,T], y \in \Omega
\end{align*}
as well as $\Phi_0^{(h)}(t,.) = \textrm{id}$ and $\det \nabla \Phi_s^{(h)}(t,.) = 1$.
%

Additionally, a simple iteration of the energy-inequality yields the following energy inequality for the problem on $[0,T]$:
\begin{align*}
& E_h(\eta^{(h)}(t_0)) + \int_0^{t_0} 2R_h(\eta^{(h)},\partial_t \eta^{(h)}) + 2A(\eta,\partial_t \eta^{(h)} - v^{(h)} \circ \eta^{(h)}) + \nu \norm[\Omega]{\varepsilon v^{(h)}}^2  + h\norm{\nabla^{k_0} v^{(h)}}^2dt \\
  &\quad+  \fint^{t_0}_{t_0-h} \frac{\rho_s}{2} \norm[Q]{\partial_t \eta^{(h)}}^2 +  \frac{\rho_f}{2}\norm[\Omega]{v^{(h)}}^2 dt + \int_0^{t_0} \inner[\Omega]{f}{v^{(h)}} dt \leq E_h(\eta_0) + \frac{\rho_s}{2}\norm[Q]{b}^2 +  \frac{\rho_f}{2}\norm[\Omega]{v_0}^2.
\end{align*}

Similarly to the estimate for the discrete approximation, this again gives us a uniform bound on $\norm[Q]{\eta_0-\eta^{(h)}(t_0)}$ in terms of $t_0$, which we can use to initially choose $T$ small enough to avoid collisions.

As a consequence, we conclude that there is a constant $C$ independent of $h$ such that
\begin{align*}
  &\sup_{t\in [0,T]} \Big(E_h(\eta^{(h)}(t)) + \norm[W^{2,q}(Q)]{\eta^{(h)}(t)}+ h^{a_0}\norm[W^{k_0,2}(Q)]{\eta^{(h)}(t)}\Big)\leq C,\\
&  \int_0^T \Big(\norm[W^{1,2}(Q)]{\partial_t \eta^{(h)}(t)}^2 + h \norm[W^{k_0,2}(Q)]{\partial_t \eta^{(h)}(t)}^2 + \norm[W^{1,2}(\Omega)]{v^{(h)}}^2 + h\norm[W^{k_0,2}(\Omega)]{v^{(h)}}^2\Big) dt \leq C, \\ &\int_0^T A(\eta,v^{(h)}\circ \eta^{(h)} - \partial_t \eta^{(h)}) dt \leq C.&&
\end{align*}
These can again be used to pick weak*-converging sub-sequences and a limit pair $(\eta,v)$. Our goal is again to prove convergence of the equation and these weak convergences are indeed enough to do so for all the terms which are unchanged from before. However this time, we have to additionaly deal with the two $h$-dependent inertial terms, of which in particular the fluid-term requires some attention as weak convergence is not enough to go to the limit.

Interestingly in this case, the correct quantities to consider are time-averages. Consider first the solid. Then we can define the rolling average of the momentum as
\begin{align*}
  m^{(h)}(t,x) = \rho_s\fint_{t-h}^t  \partial_t \eta^{(h)}(s,x) ds
\end{align*}
and note that its time derivative
\begin{align*}
 \partial_t m^{(h)}(t,x) = \rho_s\frac{\partial_t \eta^{(h)}(t)- \partial_t \eta^{(h)}(t-h)}{h}
\end{align*}
is something already occuring in the equation and over which we thus already have some weak control.

The same is true for the fluid, where we however have to be a bit more careful and instead of the physically correct quantity $\fint \rho_f v \circ \Phi$, we instead have to consider a similarly ``straightened'' version
\begin{align*}
  m^{(h)}(t,x) = \rho_f \fint_{t-h}^t v(s,x) ds 
\end{align*}
For these averaged momenta, since we have $L^2(W^{1,2})$-bounds from the energy estimate and 
bounds on their time-derivatives in negative spaces from the equation, we can apply a version of the Aubin-Lions-lemma (see e.g.\ \cite[Sec.\ 7.3]{roubcek.t:nonlinear}) to obtain their strong $L^2$-convergence (in space-time).
In particular, for the fluid this implies
\begin{align*}
 &\int_0^T \inner{\frac{v^{(h)}(t)-v^{(h)}(t-h) \circ \Phi_{-h}^{(h)}(t)}{h}}{\xi(t)} dt = -  \int_0^T \inner{v^{(h)}(t)}{\frac{\xi(t+h)\circ \Phi_{h}^{(h)}(t)-\xi(t) }{h}} dt\\
 &= -\int_0^T \inner{v^{(h)}(t)}{\fint_{0}^h \partial_s \xi(t+s) \circ \Phi_s^{(h)}(t) ds } dt \\
 &=-\int_0^T \inner{v^{(h)}(t)}{\fint_{0}^h \left[ \partial_t \xi(t+s) + v^{(h)}(t+s) \cdot \nabla \xi(t+s) \right] \circ \Phi_s^{(h)}(t) ds } dt \\
 &\to -\int_0^T \inner{v(t)}{\partial_t \xi + v \cdot \nabla \xi} dt.
\end{align*}
Next note that a Minty-type argument (compare \cite[Lem.\ 3.9]{BenKamSch20}) allows us to deal with the second order term in $DE(\eta^{(h)})$ while the convergence of the first order terms results from compact embeddings and the boundedness of the determinant. In total implies the weak convergence of $DE(\eta^{(h)})$ to $DE(\eta)$. Conclusively we then have
\begin{align*}
  &\int_0^T \inner{DE(\eta)}{\phi} +  \inner{D_2 R(\eta,\partial_t \eta)}{\phi} + \inner{D_2 A(\eta,\partial_t \eta - v \circ \eta)}{\phi - \xi \circ \eta}  \\
  &\quad+ \nu \inner{\varepsilon v}{\varepsilon \xi}- \rho_s\inner{\partial_t \eta}{\partial_t \phi} - \rho_f\inner{v}{ \partial_t \xi - v \cdot \nabla \xi} + \inner{f}{\xi} dt = 0
\end{align*}
for all $\phi\in C^\infty([0,T]\times Q;\R^n)$ and $\xi \in C_0^\infty([0,T]\times \Omega;\R^n)$ with $\phi|_{[0,T]\times P} = 0$ and $ \diver \xi = 0$, as desired.

This initially holds only on a small interval on which collisions are avoided a-priori, but by standard arguments, we can extend this interval until we reach the original $T$ or an actual collision.

Finally, we can spare some thoughts on the pressure, which can be found through the equation. For this we need the so called Bogovski\u{\i}-operator which is a bounded linear operator $\bog: W^{k,p}(\Omega) \to W_0^{k+1,p}(\Omega;\R^n)$ such that $\diver \bog \psi = \psi$. With this we can define the operator
\begin{align*}
 P(\psi) :=& \int_0^T \Big(\inner{D_2A(\eta,\partial_t \eta - v \circ \eta)}{ \bog (\psi) \circ \eta} - \nu \inner{\varepsilon v}{\varepsilon \bog (\psi)} 
 \\
 &\quad + \rho_f\inner{v}{\partial_t \bog (\psi) - v \cdot \nabla \bog (\psi)} - \inner{f}{\bog \psi}\Big) dt
\end{align*}
for which a short calculation reveals that it is a bounded distribution in space-time which gives the pressure of the fluid. Indeed, adding $\nabla p(\xi):= -P(\diver(\xi))$ to the equation allows us to test with testfunctions of non-zero divergence and shows that the solution (distributionally) satisfies the Navier-Stokes equations.

%

\subsection*{Declarations}
This work has been supported by the Primus research programme PRIMUS/19/ SCI/01, the University Centre UNCE/SCI/023 of Charles University as well as the grant GJ19-11707Y of the Czech national grant agency (GA\v{C}R). Further S.~S.\ wishes to thank the University of Vienna for their kind hospitality in winter 2020/21. There is no conflict of interest.
\noindent

\bibliographystyle{plainnat}
\bibliography{erc-bib}
 
\end{document}